\documentclass[12pt]{amsart}
\usepackage[margin=1in]{geometry}
\usepackage{amsmath, amssymb, amsthm,hyperref}
\usepackage{mathtools,xcolor}
\usepackage{enumitem}
\usepackage{multicol}
\usepackage{bbm} 

\usepackage{soul}

\newtheorem{theorem}{Theorem}[section]
\newtheorem{lemma}[theorem]{Lemma}
\newtheorem{proposition}[theorem]{Proposition}
\newtheorem{corollary}[theorem]{Corollary}
\theoremstyle{definition}

\newtheorem{example}[theorem]{Example}

\newtheorem{conjecture}[theorem]{Conjecture}
\newtheorem{remark}[theorem]{Remark}

\newcommand{\g}{\mathfrak{g}}

\numberwithin{equation}{section} 
\usepackage{color}

\definecolor{aiSuggestedPurple}{RGB}{128,0,128}

\begin{document} 

\title{On Lagrangianity of $p$-supports of holonomic D-modules and $q$-D-modules}

\author{Pavel Etingof}

\address{Department of Mathematics, MIT, Cambridge, MA 02139, USA}

\maketitle

\centerline{\bf To Maxim Kontsevich on his 60th birthday with admiration}

\begin{abstract} M. Kontsevich conjectured and T. Bitoun proved that if $M$ is a nonzero holonomic D-module then the $p$-support of a generic reduction of $M$ to characteristic $p>0$ is Lagrangian. We provide a new elementary proof of this theorem and also generalize it to $q$-D-modules. The proofs are based on Bernstein's theorem that any holonomic D-module can be transformed by an element of the symplectic group into a vector bundle with a flat connection, and a $q$-analog of this theorem. We also discuss potential applications to 
quantizations of symplectic singularities and to quantum cluster algebras. 
\end{abstract}

\tableofcontents

\section{Introduction} One of the most important basic results in noncommutative algebra is Gabber's theorem, which in particular implies that the singular support of a nonzero holonomic $D$-module on a smooth variety $X$ over a field $\bold k$ of characteristic zero is Lagrangian. M. Kontsevich conjectured and T. Bitoun proved the following refinement of this theorem: if $M$ is a nonzero holonomic $D$-module over $\bold k$ then the $p$-support of a generic reduction of $M$ to characteristic $p>0$ is Lagrangian. Later Bitoun's proof was simplified by M. van den Bergh. 

We give a new self-contained and elementary proof of this result based on Bernstein's theorem that a holonomic D-module on $\Bbb A^r$ can be transformed by an element of ${\rm Sp}(2r,\bold k)$ into a vector bundle with a flat connection $\nabla$ (Section 2). This allows us to reduce the problem to checking a Lagrangian property of $\nabla$, which can be expressed in terms of its $p$-curvature and then verified directly. At the end of the section, we discuss 
potential applications to quantizations of conical symplectic singularities. 

Then we generalize this story to $q$-D-modules (Section 3). 
We first generalize Bernstein's theorem to the $q$-case, and then develop a theory of $p$-supports for $q$-D-modules and use the method of Section 2 to prove a $q$-version of Bitoun's theorem.\footnote{Note that since for $q$-D-modules with fixed $q$ there is no notion of singular support (as there is no filtration on the algebra of $q$-difference operators with commutative associated graded algebra), considering the $p$-supports is the only available way to treat such modules semiclassically.} At the end of the section we discuss potential applications 
to quantum cluster algebras.  

{\bf Acknowledgements.} I am very grateful to Vadim Vologodsky and Michel van den Bergh for useful discussions, and also to Gwyn Bellamy, David Jordan, Ivan Losev and Milen Yakimov for discussions related to Subsections \ref{symsing} and \ref{clusalg}. The latest version was cleaned up using ChatGPT 6 Pro and the website \url{https://mathaudit.org/}; I thank Vasily Neckrasov for auditing the paper at this website. This work was partially supported by the NSF grant DMS-2001318.

\section{Lagrangianity of $p$-supports of a holonomic D-module} 

\subsection{Flat connections and $p$-curvature} 
Let $\bold k$ be a commutative unital ring and $\nabla$ be a connection on the 
trivial vector bundle over $\Bbb A^r$  of rank $n$ defined over $\bold k$.\footnote{Note that by the Quillen-Suslin theorem (formerly Serre's conjecture, see \cite{La}), if $\bold k$ is a field then every vector bundle on $\Bbb A^r$ defined over $\bold k$ is trivial.}  Thus 
$\nabla=(\nabla_1,...,\nabla_r)$, where 
$\nabla_j=\partial_j+a_j(x_1,...,x_r)$, with $a_j\in {\rm Mat}_n(\bold k[x_1,...,x_r]).$
One says that $\nabla$ is {\bf flat} if $[\nabla_i,\nabla_j]=0$, i.e., $\partial_i a_j-\partial_j a_i+[a_i,a_j]=0.$

Let $D_r(\bold k)$ be the algebra of polynomial differential operators on $\Bbb A^r$ defined over $\bold k$, i.e., the algebra 
generated by $x_i,\partial_
i,1\le i\le r$, with defining relations $[x_i,x_j]=[\partial_i,\partial_j]=0$, $[\partial_i,x_j]=\delta_{ij}$.  
Then any flat connection $\nabla$ on $\Bbb A^r$ defines a $D_r(\bold k)$-module $M_\nabla$ 
generated by $f_1,...,f_n$ with defining relations $\nabla \bold f=0$, where $\bold f=(f_1,...,f_n)^T$.

Now suppose $\bold k$ is a field of characteristic $p>0$.  
Then the $p$-{\bf curvature} $C(\nabla)$ of a flat connection $\nabla$ is the collection of operators 
$C_j=\nabla_j^p,\ 1\le j\le r$. Since $[\nabla_j,x_i]=\delta_{ij}$, we have 
$[x_i,C_j]=0$, hence $C_j\in {\rm Mat}_n(\bold k[x_1,...,x_r])$. 
Also we have $[\nabla_i,C_j]=0$, so $[C_i,C_j]=0$ and the coefficients 
of the characteristic polynomial $\chi_{\bold t}$ of $\sum_{j=1}^r t_jC_j$ 
are annihilated by all $\partial_i$. Hence they belong to $\bold k[t_1,...,t_r,X_1,...,X_r]$, where $X_i:=x_i^p$. For example, if $\nabla=\partial+a$ is a rank $1$ connection on $\Bbb A^1$
then $C(\nabla)=\partial^{p-1}a+a^p$. 

\subsection{Lagrangian flat connections in characteristic $p$}\label{lagfla} Let us say that $\nabla$ is {\bf separable} if the eigenvalues of $C_j$ lie in the separable closure $\overline{\bold k(X_1,...,X_r)}_s$. E.g., every flat connection of rank $n<p$ is automatically separable. Let $\nabla$ be separable and suppose $v$ is a common eigenvector of $C_j$, i.e. $C_jv=\Lambda_j v,\ \Lambda_j\in \overline{\bold k(X_1,...,X_r)}_s.$ Recall that the derivations $\frac{\partial}{\partial X_j}$ of $\bold k(X_1,...,X_r)$ extend in a unique way to $\overline{\bold k(X_1,...,X_r)}_s$. Let us say that $\nabla$ is {\bf Lagrangian} if for each such $v$ the 1-form $\Lambda:=\sum_{j=1}^r \Lambda_j dX_j$ is closed, i.e., $\frac{\partial \Lambda_i}{\partial X_j}=\frac{\partial \Lambda_j}{\partial X_i}$. 

The motivation for this terminology is as follows. 
 Recall that the center 
of $D_r(\bold k)$ is 
$Z:=\bold k[X_1,...,X_r,P_1,...,P_r]$, where 
$P_j:=\partial_j^p$. Let ${\rm supp}\nabla$ be the support of $M_\nabla$ as a $Z$-module, a closed subvariety in ${\rm Spec}(Z)=T^*(\Bbb A^r)^{(1)}$. 
It is clear that  ${\rm supp}\nabla$ is given in the coordinates $X_i,P_i$ by the equations
$\chi_{\bold t}(-\sum_{j=1}^r t_jP_j)=0$
$\forall \bold t=(t_1,...,t_r)\in \overline{\bold k}^r$, i.e., it consists of points $(x_1^p,...,x_r^p,-\Lambda_1,...,-\Lambda_r)$ such that 
$\sum_{j=1}^r t_j\Lambda_j$ is an eigenvalue of $\sum_{j=1}^r t_jC_j(x_1,...,x_r)$ for all $\bold t$. Thus, recalling that the graph in $T^*Y$ of a 1-form $\omega$ on a smooth variety $Y$ is Lagrangian if and only if $\omega$ is closed, we obtain

\begin{proposition}\label{lagcri} $\nabla$ is Lagrangian if and only if ${\rm supp}\nabla$ is a Lagrangian subvariety 
of $T^*(\Bbb A^r)^{(1)}$.
\end{proposition}

Analogous definitions and statements apply to connections on the formal polydisk $D^r$, 
with the polynomial algebra $\bold k[x_1,...,x_r]$ replaced by the formal series algebra $\bold k[[x_1,...,x_r]]$. They also apply to connections on a smooth affine variety $U$ equipped with coordinates $x_1,...,x_r$, with the polynomial algebra $\bold k[x_1,...,x_r]$ replaced by $\bold k[U]$. By gluing such charts $U$, the story extends to connections
on arbitrary vector bundles on any smooth variety. 
 
\subsection{Liftable connections} 
In general, a separable flat connection, even one of rank $1$, need not be Lagrangian. 

\begin{example}\label{nonlag} Let $\nabla$ be the flat connection of rank $1$ on $\Bbb A^2$ over $\Bbb F_p$ with 
$$
\nabla_1=\partial_1,\ \nabla_2=\partial_2+x_1^{p}x_2^{ p-1}.
$$  
Then the $p$-curvature of $\nabla$ is $C_1=0,C_2=-X_1+X_1^{p}X_2^{p-1}$, so $\nabla$ is not Lagrangian. 
\end{example}

However, the Lagrangian property holds under some assumptions.   
Namely, say that a flat connection $\nabla$ defined over $\bold k$ is {\bf liftable} if it is the reduction to $\bold k$ of a flat connection defined over 
some local Artinian ring $S$ with residue field $\bold k$ in which $p\ne 0$. 

\begin{lemma}\label{l1} If a flat connection of rank $1$ on a formal polydisk $D^r$ 
is liftable then it is Lagrangian. 
\end{lemma} 

\begin{proof} It suffices to prove the statement for $r=2$. In this case we have 
$\nabla=(\nabla_x,\nabla_y)$ with 
$\nabla_x=\partial_x+a(x,y)$ and $\nabla_y=\partial_y+b(x,y)$
and lifting $\widetilde\nabla=(\widetilde\nabla_x,\widetilde\nabla_y)$ with 
$\widetilde\nabla_x=\partial_x+\widetilde a(x,y)$ and $\widetilde\nabla_y=\partial_y+\widetilde b(x,y)$. Let $a_{ij},b_{ij},\widetilde a_{ij},\widetilde b_{ij}$ be the coefficients 
of $a,b,\widetilde a,\widetilde b$. The zero curvature equation $\widetilde a_y=\widetilde b_x$
yields 
\begin{equation}\label{zercu}
j\widetilde a_{i-1,j}=i\widetilde b_{i,j-1},\ i,j\ge 1.
\end{equation}
On the other hand, the $p$-curvature of $\nabla$ is $C(\nabla)=(A,B)$, where
$$
A(X,Y)=\partial_x^{p-1}a(x,y)+a^{(1)}(X,Y),\ B(X,Y)=\partial_y^{p-1}b(x,y)+b^{(1)}(X,Y),
$$
where $X=x^p,Y=y^p$, and $a^{(1)},b^{(1)}$ are the Frobenius twists of $a,b$ (all coefficients raised to power $p$). Now, reducing equation \eqref{zercu} modulo $p$ and then raising it to $p$-th power, we get $ja_{i-1,j}^p=ib_{i,j-1}^p,\ i,j\ge 1$, hence 
$$
\partial_Y a^{(1)}(X,Y)=\partial_X b^{(1)}(X,Y).
$$
So it remains to show that 
$$
\partial_Y \partial_x^{p-1}a(x,y)=\partial_X \partial_y^{p-1}b(x,y),
$$
which is equivalent to the equations
$$
ja_{pi-1,pj}-ib_{pi,pj-1}=0,\ i,j\ge 1.
$$
But ~\eqref{zercu} gives $p(j\widetilde a_{pi-1,pj}-i\widetilde b_{pi,pj-1})=0$.
Since $p\ne0$ in the local ring $S$, the expression in parentheses cannot
be a unit.  It therefore belongs to the maximal ideal of $S$, and its
reduction is zero, as required.
\end{proof} 

On the other hand, a liftable separable flat connection of rank $n\ge p$ need not be Lagrangian. 

\begin{example}\label{LMex} Let $\sigma=(12...p)$ be a cyclic permutation matrix, $L={\rm diag}(0,1,...,p-1)$, so $\sigma^{-1} L\sigma=L+1-pE_{pp}$. Let $\widetilde\nabla$ be the connection of rank $p$ on $\Bbb A^2$ over $\Bbb Z[\sqrt{p}]/p^2$ with 
$$
\widetilde\nabla_1=\partial_1-\sqrt{p}x_1^{p-1}\sigma^{-1}L,\ \widetilde\nabla_2=\partial_2+x_1^px_2^{p-1}+\sqrt{p} x_2^{p-1}\sigma;
$$ 
it is easy to check that it is flat. But the reduction $\nabla$ of $\widetilde \nabla$ to $\bold k$ is the direct sum of $p$ copies of the connection of Example \ref{nonlag}, 
which is not Lagrangian (even though liftable).
\end{example} 

However, there are no such examples for $n<p$. 

\begin{lemma}\label{c2} Let $\nabla$ be a liftable flat connection of rank $n<p$ on $D^r$. Assume that the joint eigenvalues of the $p$-curvature of $\nabla$ at $0$ and at the generic point of $D^r$ have the same multiplicities. Then $\nabla$ is Lagrangian. 
\end{lemma} 

\begin{proof} We may assume that $p^2=0$ in $S$. 
Let $\widetilde\nabla$ be a lift of $\nabla$ over $S$. Let 
$\Lambda^{(k)}=\sum_{j=1}^r \Lambda^{(k)}_jdX_j$ be the distinct eigenvalue 1-forms of $C(\nabla)$ and $m_k$ be their multiplicities. The constant multiplicity condition implies that, after replacing $S$ by its finite \'etale local extension if needed, we have a decomposition $\widetilde\nabla\cong \oplus_k \widetilde\nabla^{(k)}$, where $\widetilde\nabla^{(k)}$ is a flat connection of rank $m_k$ whose reduction $\nabla^{(k)}$ over the residue field has a single $p$-curvature eigenvalue $\Lambda^{(k)}$. Namely, let $\lbrace {\bold e}_k\rbrace$ be the complete system of $C_i=\nabla_i^p$-invariant orthogonal idempotents projecting to its generalized eigen-bundles with eigenvalues $\Lambda^{(k)}$, and let $\lbrace\widetilde {\bold e}_k\rbrace$ be its lift over $S$ to a complete system of orthogonal idempotents invariant under $\widetilde\nabla_i^{p^2}$; then we have $\widetilde\nabla^{(k)}={\rm Im}\widetilde {\bold e}_k$ and $\nabla^{(k)}={\rm Im}{\bold e}_k$.

Since ${\rm supp}(\oplus_k\nabla^{(k)})=\cup_k {\rm supp}\nabla^{(k)}$, we may thus assume without loss of generality that $C(\nabla)$ has a single eigenvalue 1-form $\Lambda$. Then 
$$
C_i(\wedge^n\nabla)=\sum_{j=1}^n 1^{\otimes j-1}\otimes C_i(\nabla)\otimes 1^{\otimes n-j}|_{\wedge^n\bold k^n} ={\rm Tr}(C_i(\nabla))=n\Lambda_i.
$$ 
Since the connection $\wedge^n\nabla$ is liftable and has rank $1$ and $n<p$, 
Lemma \ref{l1} implies that $d\Lambda=0$, as desired. 
\end{proof} 

\begin{corollary}\label{c3} If a flat connection $\nabla$ of rank $n<p$ on a smooth variety $X$ defined over $\bold k$ is liftable, then $\nabla$ is Lagrangian. 
\end{corollary}

\begin{proof} It suffices to prove the corollary in the formal neighborhood of a generic point of $X$. But this follows from Lemma \ref{c2}, since the constant multiplicity assumption of this lemma holds at a generic point. 
\end{proof} 

\subsection{Bernstein's theorem} 
Let $\bold k$ be an algebraically closed field and $V$ be a symplectic vector space over $\bold k$ of dimension $2r$. Let $Y\subset V$ be an affine cone of dimension $r$. 

\begin{lemma}\label{l4}(J. Bernstein; \cite{AGM}, Lemma 3.15) (i) There exists a Lagrangian subspace $L\subset V$ 
such that $L\cap Y=\lbrace 0\rbrace$. 

(ii) If $V=E\oplus E^*$ with the standard symplectic form, then there exists 
$g\in {\rm Sp}(V)$ such that the natural projection $gY\to V/E^*\cong E$ is finite.  
\end{lemma} 

\begin{proof} (i) Let ${\rm LGr}(V)$ be the Lagrangian Grassmannian of $V$, 
then 
$$
{\rm LGr}(V)\cong {\rm LGr}_r:={\rm Sp}(2r,\bold k)/(GL(r,\bold k)\ltimes S^2\bold k^r),
$$ 
so $\dim {\rm LGr}(V)=\dim {\rm LGr}_r=\frac{r(r+1)}{2}$. Consider the closed subvariety 
$X\subset \Bbb PY\times {\rm LGr}(V)$ consisting of pairs $(y,L)$, $y\in \Bbb PY$, $L\in {\rm LGr}(V)$ such that $y\subset L$. The fiber $p^{-1}(y)$ of the projection $p: X\to \Bbb PY$ consists of all $L\in {\rm LGr}(V)$ containing $y$, so it is isomorphic to ${\rm LGr}(y^\perp/y)\cong {\rm LGr}_{r-1}$ via $L\mapsto L/y$. Thus we get 
$$
\dim X=\dim \Bbb PY+\dim {\rm LGr}_{r-1}=r-1+\frac{r(r-1)}{2}=\frac{r(r+1)}{2}-1.
$$ 
So the projection $\pi: X\to {\rm LGr}(V)$ is not surjective for dimension reasons, i.e., there exists $L\in {\rm LGr}(V)$ not containing any $y\in \Bbb PY$. Then $L\cap Y=\lbrace 0\rbrace$.  

(ii) By (i) there is a Lagrangian subspace $L\subset V$ such that $L\cap Y=\lbrace 0\rbrace$. 
There exists $g\in {\rm Sp}(V)$ such that $gL=E^*$. 
Then $E^*\cap gY=\lbrace 0\rbrace$, so the projection $gY\to V/E^*$ is finite. 
\end{proof} 

Now assume that ${\rm char}(\bold k)=0$ and let $M$ be a nonzero holonomic $D_r(\bold k)$-module. The algebra $D_r(\bold k)$ is the Weyl algebra of the symplectic space $V=\bold k^r\oplus (\bold k^r)^*=\bold k^{2r}$, hence the group ${\rm Sp}(2r,\bold k)$ acts on $D_r(\bold k)$ by automorphisms. 
For $g\in {\rm Sp}(2r,\bold k)$ let $gM$ be the twist of  
$M$ by the automorphism $g$ of $D_r(\bold k)$. 
 
\begin{theorem}\label{c5} (J. Bernstein; \cite{AGM}, Corollary 3.16) 
There exists $g\in {\rm Sp}(2r,\bold k)$ such that $gM$
is $\mathcal O$-coherent (a vector bundle on $\bold k^r$ with a flat connection).  
\end{theorem} 

\begin{proof} Let $Y:=SS_a(M)\subset V$ be the arithmetic singular support of $M$, i.e. the support of ${\rm gr}_FM$, where $F$ is a good filtration on $M$ with respect to the Bernstein filtration on $D_r(\bold k)$ ($\deg(x_j)=\deg(\partial_j)=1$). Then $Y$ is an affine cone of dimension $r$. Let $g$ be chosen as in Lemma \ref{l4}(ii). 
Then $gY\to\bold k^r$ is finite, so the associated graded module of
$gM$ with its transported filtration, and hence $gM$ itself, is finite
over $\mathcal O(\bold k^r)$.
\end{proof} 

\subsection{Bitoun's theorem} 
Now let $\bold k$ be an algebraically closed field of characteristic $0$ and
$M$ be a finitely generated 
$D_r(\bold k)$-module. Pick a finitely generated subring $R\subset \bold k$ and a finitely generated $D_r(R)$-module $M_R$ with an isomorphism $\bold k\otimes_R M_R\cong M$. For every prime $p$ which is not a unit in $R$ 
and a maximal ideal $\mathfrak m\subset R$ lying over $p$ (i.e., the preimage of a maximal ideal $\overline{\mathfrak m}\subset R/pR$) with residue field 
$k_{\mathfrak m}=R/\mathfrak m$ let $M_{\mathfrak m}$ be the finitely generated $D_r(k_{\mathfrak m})$-module $k_{\mathfrak m}\otimes_R M_R$. Let $Z_{\mathfrak m}$ be the center of $D_r(k_{\mathfrak m})$. 
The support ${\rm supp}(M_{\mathfrak{m}})$ of $M_{\mathfrak m}$ in ${\rm Spec}(Z_{\mathfrak m})$ is called the $p$-{\bf support} of $M$ at $\mathfrak{m}$.

The following result was conjectured by M. Kontsevich (\cite{K}, Conjecture 2) and proved by T. Bitoun \cite{B}. The proof was later simplified by M. van den Bergh \cite{vdB}. 

\begin{theorem}\label{t6} If $M\ne 0$ is a holonomic module then the $p$-support ${\rm supp}(M_{\mathfrak{m}})$ is Lagrangian for generic $\mathfrak m\in {\rm Spec} R$.
\end{theorem} 

\begin{example} If $\bold k=\overline{\Bbb Q}$ then there are finitely many 
$\mathfrak{m}$ over each prime $p$, so Theorem \ref{t6} holds for almost all $p$.  
\end{example} 

\begin{proof} By Theorem \ref{c5}, it suffices to prove the theorem when $M$ is a vector bundle with a flat connection $\nabla$. In this case, 
the result follows from Corollary \ref{c3}. 
\end{proof} 

As explained in \cite{B},\cite{vdB},\ Theorem \ref{t6} implies its generalization to any smooth variety, which is the formulation conjectured in \cite{K} and proved in \cite{B}. 

\begin{corollary}\label{t7} (\cite{B}) If $X$ is any smooth variety over $\bold k$ and $M$ a nonzero holonomic $D(X)$-module, 
then $\mathcal Z_{\mathfrak m}:={\rm supp}(M_{\mathfrak m})\subset T^*X^{(1)}$ is Lagrangian for generic $\mathfrak{m}\in {\rm Spec} R$. 
\end{corollary} 
 
\begin{proof} It suffices to treat the case of affine $X$. Pick a closed embedding $i: X\hookrightarrow E=\Bbb A^r$ defined over $R$ and let $\widetilde{\mathcal Z}_{\mathfrak{m}}:=\iota(\pi^{-1}({\mathcal Z}_{\mathfrak{m}}))$, 
where $$\pi:  X^{(1)}\times E^{(1)*}\twoheadrightarrow T^*X^{(1)},\quad \iota=i^{(1)}\times {\rm id}:  X^{(1)}\times E^{(1)*}\hookrightarrow  E^{(1)}\times E^{(1)*}$$ are the natural maps. Then 
${\rm supp}((i_*M)_{\mathfrak{m}})=\widetilde{\mathcal Z}_{\mathfrak{m}}$, so it is Lagrangian  
by Theorem \ref{t6}. Hence so is $\mathcal Z_{\mathfrak{m}}$. 
\end{proof} 

Note that Corollary \ref{t7} implies Gabber's theorem that the 
singular support $SS(M)$ of a nonzero holonomic D-module $M$ on any smooth variety is Lagrangian, since the singular support of $M$ is obtained by conical degeneration of the $p$-support for large $p$. 

\begin{remark}\label{adddif} Let ${\rm char}(\bold k)=0$ and $\mathcal D_r^{\rm add}(\bold k)$ be the algebra of additive difference operators, generated by $x_i,T_i^{\pm 1}$, $1\le i\le r$, with defining relations 
$$
[x_i,x_j]=[T_i,T_j]=0,\ T_ix_j=(x_j+\delta_{ij})T_i.
$$
This algebra acts naturally on $\bold k[x_1,...,x_r]$, with $x_i$ acting by multiplication 
and $T_i$ acting by translation operators
$$
(T_if)(x_1,...,x_r)=f(x_1,...,x_i+1,...,x_r).
$$ 
The algebra $\mathcal D_r^{\rm add}(\bold k)$ is isomorphic
to the algebra $D((\bold k^\times)^r)$ of differential operators on the torus
$(\bold k^\times)^r$ with coordinates $y_i$, via $T_i\mapsto y_i,x_i\mapsto -y_i\frac{\partial}{\partial y_i}$.  Thus Corollary \ref{t7} applies directly to $\mathcal D_r^{\rm add}(\bold k)$-modules.
\end{remark} 

\begin{remark} Besides providing a new proof of Bitoun's theorem, our approach 
provides a simplification of its previous proofs of \cite{B},\cite{vdB}. Namely, both proofs rely on Theorem 4.3.3 of \cite{B} (Theorem 3.4 in \cite{vdB}), which 
is easy for $\mathcal{O}$-coherent D-modules $M$ (the number $e(M)$ in the last theorem can be taken to be the rank of the vector bundle). Thus this theorem follows immediately from 
Theorem \ref{c5}. 
\end{remark} 

\subsection{$p$-supports of holonomic modules over quantized conical symplectic singularities}\label{symsing}

Recall that a {\it conical symplectic singularity} over $\Bbb C$ is an irreducible affine conical Poisson $\Bbb C$-variety $X$ such that $X$ is normal with symplectic smooth locus $X^\circ$ and its symplectic form extends to a regular 2-form on some (equivalently, any) resolution of singularities of $X$ (\cite{Be}). Namikawa showed that $X$ is rigid under homogeneous Poisson deformations (i.e., with a fiberwise $\Bbb G_m$-action), hence 
always defined over $\overline{\Bbb Q}$ (\cite{N1}), and Kaledin showed that $X$ has finitely many symplectic leaves (\cite{K1}). 

It is known that a conical symplectic singularity $X$ admits a homogeneous partial resolution of singularities $\phi: \widetilde X\to X$ called a $\Bbb Q$-{\it factorial terminalization}, such that the smooth locus $\widetilde X^\circ\subset \widetilde X$ has complement of codimension $\ge 4$ and the symplectic form of $X^\circ\subset \widetilde X$ extends to one on $\widetilde X^\circ$ (this can be deduced from the main result of \cite{BCHM}, see \cite{Lo3}, Proposition 2.1). Put $\mathfrak h=H^2(\widetilde X,\Bbb C)$. Namikawa constructed a
universal graded Poisson deformation of $X$ over $\mathfrak h/W$, where
$W$ is the Namikawa Weyl group (\cite{N2,N4}). We denote by $X_\lambda$,
$\lambda\in\mathfrak h$, the fibers of its pullback along
$\mathfrak h\to\mathfrak h/W$; these are filtered Poisson deformations
of $X$. Moreover, all $X_\lambda$ have symplectic singularities (in particular, finitely many symplectic leaves). 

Furthermore, Losev showed in \cite{Lo1} that this deformation lifts to the quantum level, giving a family of filtered algebras $A_{t,\lambda}$, $t\in \Bbb C$, 
such that $A_{0,\lambda}=\Bbb C[X_\lambda]$ and ${\rm gr}(A_{t,\lambda})=
\Bbb C[X]$. Namely, this family is obtained as the graded part of the completed algebra 
$\widehat A_{t,\lambda}=H^0(\widetilde X^\circ,\mathcal A_{t,\lambda})$, where 
$\mathcal A_{t,\lambda}$ is a sheaf of deformation quantizations of $\widetilde X^\circ$ defined by Bezrukavnikov and Kaledin (\cite{BK}). 
The family $A_{t,\lambda}$ does not depend on the choice of
$\widetilde X$, and $A_{t,\lambda}\cong A_{st,s\lambda}$ for
$s\in\Bbb C^\times$. Every filtered quantization of $X$ is isomorphic
to some $A_{1,\lambda}$, and $A_{1,\lambda}\cong A_{1,\lambda'}$
as filtered quantizations if and only if $\lambda'\in W\lambda$
(\cite[Theorem~3.4]{Lo1}). We denote $A_{1,\lambda}$ by $A_\lambda$.

Moreover, it is strongly expected and in many cases known\footnote{Possibly this follows in general from the known 
results, but I was unable to find a suitable reference.} that the reduction $A_{\lambda,p}$ of $A_\lambda$ to a field $\bold k$ of characteristic $p$ (defined for sufficiently large $p$)\footnote{Here and below
``for sufficiently large $p$'' means for maximal ideals in a nonempty open 
subset of $\operatorname{Spec}R$, after a suitable localization of $R$.} has a big center $Z_{\lambda,p}$ such that  ${\rm gr}Z_{\lambda,p}=\bold k[X]^p=\bold k[X^{(1)}]$, where $X^{(1)}$ is the Frobenius twist of $X$ (in particular, it is module-finite over the center), and ${\rm Spec}Z_{\lambda,p}$ is isomorphic to $X_{\lambda^p-\lambda}^{(1)}$.\footnote{If $\lbrace c_j\rbrace$ is a basis of $H^2(\widetilde X,\Bbb Z)$ and $\lambda=\sum_i \lambda_ic_i$, $\lambda_i\in \bold k$, then $\lambda^p:=\sum_i \lambda_i^p c_i$.} Namely, this is satisfied for quotient singularities (by the Appendix to \cite{BFG}), for Higgs branches, such as quiver varieties (by virtue of the hamiltonian reduction procedure), for the Springer resolution for a simple Lie algebra $\g$ (by the structure of the $p$-center of $U(\g)$ in characteristic $p$), hence for Slodowy slices, for Coulomb branches (by \cite{L}), etc. 
  
Thus to every finitely generated module $M$ over $A_{\lambda}$, we can attach its sequence of $p$-supports ${\rm supp}(M_\mathfrak{m})\subset X_{\lambda^p-\lambda}^{(1)}$ which are defined for sufficiently large $p$ and independent for almost all $\mathfrak m$ on the details of reduction.\footnote{Independence of models is understood after passing to a common ring
of definition and shrinking; it does not mean independence of the
specialization $\mathfrak m$.} 

Let $Y$ be a Poisson variety with finitely many symplectic leaves. 
We say that a subvariety $Z\subset Y$ is {\bf Lagrangian} if it is coisotropic and 
the intersection of $Z$ with every symplectic leaf is isotropic.  
Also recall from \cite{Lo2} that a finitely generated $A_{\lambda}$-module $M$ is called {\it holonomic} if ${\rm supp}M$ is Lagrangian, i.e. if the intersection of ${\rm supp}M$ with every symplectic leaf of $X$ is isotropic (note that ${\rm supp}M$ is automatically coisotropic by Gabber's theorem). The following conjecture gives a potential generalization of Bitoun's theorem to conical symplectic singularities. 

\begin{conjecture}\label{coo} If $M$ is holonomic then ${\rm supp}(M_{\mathfrak m})$ is Lagrangian in $X_{\lambda^p-\lambda}^{(1)}$ for sufficiently large $p$ and almost all $\mathfrak{m}$ lying over $p$. 
\end{conjecture} 

We expect that if $X$ is a Higgs branch (i.e., obtained by Hamiltonian reduction 
from an action of a connected reductive group $G$ on a symplectic 
representation $V$, with flat moment map) then this conjecture 
can be deduced from Bitoun's theorem by Hamiltonian reduction.  

\section{Lagrangianity of $p$-supports of a holonomic $q$-D-module} 

\subsection{Flat $q$-connections and $q$-D-modules} 
In this section we generalize Bitoun's theorem to $q$-D-modules. To this end, for $q\in \bold k^\times$, define the algebra $\mathcal D_{r,q}(\bold k)$ generated by $x_i^{\pm 1},T_i^{\pm 1}$, $1\le i\le r$, with defining relations 
$$
[x_i,x_j]=[T_i,T_j]=0,\ T_ix_j=q^{\delta_{ij}}x_jT_i.
$$
This algebra acts naturally on $\bold k[x_1^{\pm 1},...,x_r^{\pm 1}]$, with $x_i$ acting by multiplication 
and $T_i$ acting by translation operators
$$
(T_if)(x_1,...,x_r)=f(x_1,...,qx_i,...,x_r).
$$ 
Thus $\mathcal D_{r,q}(\bold k)$ may be viewed as the algebra of $q$-difference operators (at least when $q$ is not a root of $1$). Unlike the case of additive difference operators (Remark \ref{adddif}), this algebra is not isomorphic to the algebra of differential operators on any variety, so the results of the previous sections 
don't apply and we have to start from scratch. 

The basic theory of $\mathcal D_{r,q}(\bold k)$-modules (a.k.a. $q$-{\bf D-modules}) for $q$ not a root of $1$ parallel to the Bernstein-Kashiwara theory of usual D-modules was developed in the first two sections of \cite{S}. Note that in these sections it is assumed that ${\rm char}(\bold k)=0$, 
but this assumption is not used, so we will not make it. 

Let us summarize the basic properties of $\mathcal D_{r,q}(\bold k)$ and its modules, mostly following \cite{S}. If $q$ is not a root of unity, then $\mathcal D_{r,q}(\bold k)$ is simple (\cite{S}, Proposition 1.2.1) and has trivial center. On the other hand, if $q$ is a root of unity of order $N$ then the center of $\mathcal D_{r,q}(\bold k)$ is 
$$
Z=\bold k[X_1^{\pm 1},...,X_r^{\pm 1},P_1^{\pm 1},...,P_r^{\pm 1}],
$$ where $X_i:=x_i^N,P_i:=T_i^N$, and it carries a natural symplectic structure induced by varying $q$, with Poisson bracket given by 
$$
\lbrace X_i,X_j\rbrace=\lbrace P_i,P_j\rbrace=0,\ \lbrace P_i,X_j\rbrace=\delta_{ij}P_iX_j.
$$ 
In this case, if $M$ is a finitely generated $\mathcal D_{r,q}(\bold k)$-module then we can define its {\bf support} ${\rm supp}(M)$ as a $Z$-module, which is a closed subvariety of the $2r$-dimensional symplectic torus $\Bbb  G_m^{2r}={\rm Spec}(Z)$.

The algebra $\mathcal D_{r,q}(\bold k)$ has Gelfand-Kirillov dimension $2r$, and if $q$ is not a root of unity, it satisfies {\bf Bernstein's inequality}: every nonzero finitely generated module has Gelfand-Kirillov dimension $\ge r$ (\cite{S}, Corollary 2.1.2). In this case, finitely generated modules of Gelfand-Kirillov dimension $\le r$ (i.e., 
of dimension exactly $r$ and the zero module) are called {\bf holonomic} (\cite{S}, Section 2). 

Examples of holonomic modules are obtained from $q$-connections with rational coefficients. Namely, a (rational) $q$-{\bf connection} $\nabla$ on $(\bold k^\times)^r$
is a collection of elements  
$$
\nabla_i=T_i^{-1}a_i,\ a_i\in GL_n(\bold k(x_1,...,x_r)).
$$
A $q$-connection $\nabla$ is said to be {\bf flat} if 
$[\nabla_i,\nabla_j]=0$, i.e., 
$$
a_i(x_1,...,qx_j,...,x_r)a_j(x_1,...,x_r)=a_j(x_1,...,qx_i,...,x_r)a_i(x_1,...,x_r).
$$
It is clear that a flat $q$-connection is the same thing as 
a $\mathcal D_{r,q}(\bold k)$-module structure on $\bold k(x_1,...,x_r)^n$ extending the obvious $\bold k[x_1^{\pm 1},...,x_r^{\pm 1}]$-module structure. It is shown in \cite{S}, Corollary 2.5.2 that if $M$ is a holonomic $\mathcal D_{r,q}(\bold k)$-module then 
the vector space 
$$
M_{\rm loc}:=\bold k(x_1,...,x_r)\otimes_{\bold k[x_1^{\pm 1},...,x_r^{\pm 1}]}M
$$ 
over $\bold k(x_1,...,x_r)$ is finite dimensional, so $M_{\rm loc}$ 
gives rise to a flat $q$-connection $\nabla_M$ defined uniquely up to conjugation by $GL_n(\bold k(x_1,...,x_r))$. (It may, of course, happen that $M\ne 0$ but $M_{\rm loc}=0$). Conversely, for any flat $q$-connection $\nabla$ there exists 
a holonomic $\mathcal D_{r,q}(\bold k)$-module $M$ such that $\nabla_M=\nabla$ (\cite{S}, Theorem 2.7.1). 

\subsection{Freeness of $\mathcal O$-coherent $q$-D-modules}

Let $f\in \bold k[x_1^{\pm 1},...,x_r^{\pm 1}]$. 
For every vector $\bold m:=(m_1,...,m_r)\in \Bbb Z^r$, let $f_{\bold m}:=T_{1}^{m_1}...T_r^{m_{r}}f$. 

\begin{lemma}\label{l10} If $f\ne 0$ then there exist $h_\bold m\in  \bold k[x_1^{\pm 1},...,x_r^{\pm 1}]$, almost all zero, 
such that $\sum_{\bold m}h_{\bold m}f_{\bold m}=1$. 
\end{lemma} 

\begin{proof} Since $q$ is not a root of unity, every orbit of $q$-translations of coordinates is Zariski dense in $(\bold k^\times)^{r}$. Thus the system of equations $f_{\bold m}=0$ for all $\bold m\in \Bbb Z^r$ has no solutions. Hence the result follows from the Nullstellensatz. 
\end{proof} 

\begin{proposition}\label{freeness} 
Let $M$ be a holonomic $\mathcal D_{r,q}(\bold k)$-module 
finite over $\bold k[x_1^{\pm 1},...,x_r^{\pm 1}]$. 
Then $M$ is a free $\bold k[x_1^{\pm 1},...,x_r^{\pm 1}]$-module. 
\end{proposition} 

\begin{proof}
Let $0\ne f\in \bold k[x_1^{\pm 1},...,x_r^{\pm 1}]$
be such that $M[1/f]$ is a free $\bold k[x_1^{\pm 1},...,x_r^{\pm 1}][1/f]$-module. Then for all $\bold m$, 
$M[1/f_{\bold m}]$ is a free 
$\bold k[x_1^{\pm 1},...,x_r^{\pm 1}][1/f_{\bold m}]$-module. 
By Lemma \ref{l10}, this means that $M$ is locally free over $\bold k[x_1^{\pm 1},...,x_r^{\pm 1}]$. But every finite locally free module over a Laurent polynomial algebra over a field is free by Swan's theorem (\cite{La}, Corollary 4.10).
\end{proof} 

\subsection{$q$-analog of Bernstein's theorem} 

Note that similarly to the differential case, the group ${\rm Sp}(2r,\Bbb Z)$ acts by automorphisms on $\mathcal D_{r,q}(\bold k)$ (\cite{S}, Subsection 1.3). Part (iii) of the following theorem is a $q$-analog of Theorem \ref{c5}. 

\begin{theorem}\label{t2} Let $M$ be a 
holonomic $\mathcal D_{r,q}(\bold k)$-module which is finite over the subalgebra $\mathcal D_{s,q}(\bold k[x_{s+1}^{\pm 1},...,x_r^{\pm 1}])$. Then 

(i) there exists $g\in {\rm Sp}(2s,\Bbb Z)$ such that $gM$ is finite over $\bold k[x_1^{\pm 1},...,x_r^{\pm 1}]$; 

(ii) $M$ is a free $\bold k[x_{s+1}^{\pm 1},...,x_r^{\pm 1}]$-module. 

(iii) For every holonomic $\mathcal D_{r,q}(\bold k)$-module $M$ there is $g\in {\rm Sp}(2r,\Bbb Z)$ such that $gM$ is finite over $\bold k[x_1^{\pm 1},...,x_r^{\pm 1}]$.
\end{theorem} 

The rest of this subsection is dedicated to the proof of Theorem \ref{t2}.

Let $R$ be a commutative domain and consider the algebra $\mathcal D_{s,q}(R)$ with generators 
$x_i^{\pm 1},T_i^{\pm 1}$, $1\le i\le s$. 

\begin{lemma}\label{l1b} Let $M$ be a finite module over $\mathcal D_{s,q}(R)$ which is torsion-free as an $R$-module. If $M$ does not contain a nonzero free submodule 
then there exists a nonempty Zariski open subset 
$U\subset {\rm Sp}(2s,\Bbb C)$ such that for any  
$g\in {\rm Sp}(2s,\Bbb Z)\cap U$ there exists nonzero $f\in R$ for which the module $gM[1/f]$   
is finite over $\mathcal D_{s-1,q}(R[1/f][x_s^{\pm 1}])\subset \mathcal 
D_{s,q}(R[1/f])$. 
\end{lemma}

\begin{proof} Let $v_1,...,v_k$ be generators of $M$ and $M_1,...,M_k$ be the cyclic submodules of $M$ generated by $v_1,...,v_k$. It suffices to prove the statement for each $M_i$, and every $M_i$ satisfies the assumptions of the lemma. Thus we may assume without loss of generality that $M\ne 0$ is cyclic on some generator $v$: indeed, if $U_i$ are the open sets corresponding to $M_i$ then we can take $U=\cap_i U_i$, and if $g\in {\rm Sp}(2s,\Bbb Z)\cap U$ and $f_i\in R$ correspond to $M_i,g$ then we may take $f=\prod_i f_i$ for $M$. 

Pick a nonzero $D\in \mathcal D_{s,q}(R)$ such that $Dv=0$ (it exists since $M$ is not free). Since $M$ is torsion-free over $R$, 
$D$ contains at least two monomials. Let 
$U$ be the open set of $g\in {\rm Sp}(2s,\Bbb C)$ 
such that $g\mu$ for all monomials $\mu$ occurring in $D$ 
have different last coordinates (where we view Laurent monomials in the generators $x_i,T_i$ as elements of $\Bbb Z^{2s}$). The module $gM$ 
is then generated by an element $w$ satisfying the equation $g(D)w=0$, which can be written as
$$
(a_0\mu_0+...+a_{n-1}\mu_{n-1}T_s^{n-1}+a_n\mu_nT_s^n)w=0,
$$
where $n\ge 1$, $a_i\in R$, $a_0,a_n\ne 0$, and 
$\mu_i$ are monomials in $x_i^{\pm 1}$, $1\le i\le s$ and $T_i^{\pm 1}$, $1\le i\le s-1$. This equation multiplied by powers of $T_s$ can be used to express $T_s^jw$, $j\in \Bbb Z$ in terms of $w,T_sw,...,T_s^{n-1}w$ with coefficients whose denominators 
are powers of $a_n$ if $j\ge n$ and powers of $a_0$ if $j<0$. 
Thus, setting $f:=a_0a_n$, we see that for every $u\in M$ there exists $N$ and $D_0,...,D_{n-1}\in \mathcal D_{s-1,q}(R[x_s^{\pm 1}])$ with 
$$
f^Nu=\sum_{i=0}^{n-1}D_iT_s^iw.
$$
This implies the statement. 
\end{proof}

\begin{lemma}\label{l1c} Let $M$ be a $\mathcal D_{r,q}(\bold k)$-module finite over 
$\mathcal D_{s,q}(\bold k[x_{s+1}^{\pm 1},...,x_r^{\pm 1}])\subset \mathcal D_{r,q}(\bold k)$. Then $M$ is torsion-free over $\bold k[x_{s+1}^{\pm 1},...,x_r^{\pm 1}]$.
\end{lemma} 

\begin{proof} Let $M_{\rm tors}\subset M$ be the subspace 
of $\bold k[x_{s+1}^{\pm 1},...,x_r^{\pm 1}]$-torsion elements. 
Then $M_{\rm tors}$ is a $\mathcal D_{r,q}(\bold k)$-submodule, and it is finite over $\mathcal D_{s,q}(\bold k[x_{s+1}^{\pm 1},...,x_r^{\pm 1}])$ since this algebra is Noetherian. Let $v_1,...,v_k$ be generators of $M_{\rm tors}$ 
over $\mathcal D_{s,q}(\bold k[x_{s+1}^{\pm 1},...,x_r^{\pm 1}])$. There exists $f\in \bold k[x_{s+1}^{\pm 1},...,x_r^{\pm 1}]$ such that $fv_i=0$ for all $i$, so $fM_{\rm tors}=0$. 
 
It follows that 
$f_{\bold m}M_{\rm tors}=0$ for all $\bold m$. 
By Lemma \ref{l10}, there exist 
 $h_{\bold m}$, almost all zero, 
 such that $\sum_{\bold m}h_{\bold m}f_{\bold m}=1$. 
Thus $M_{\rm tors}\subset \sum_{\bold m}h_{\bold m}(f_{\bold m}M_{\rm tors})=0$. 
\end{proof}  

\begin{corollary}\label{c1d} Let $M$ be a  
$\mathcal D_{r,q}(\bold k)$-module which is finite over 
$\mathcal D_{s,q}(\bold k[x_{s+1}^{\pm 1},...,x_{r}^{\pm 1}])$
for some $0<s\le r$ and does not contain a nonzero free submodule over this subalgebra. Then  
there exists $g\in {\rm Sp}(2s,\Bbb Z)$ such that 
$gM$ is finite over $\mathcal D_{s-1,q}(\bold k[x_{s}^{\pm 1},...,x_{r}^{\pm 1}])$.
\end{corollary} 

\begin{proof} By Lemma \ref{l1c}, $M$ is torsion-free over $\bold k[x_{s+1}^{\pm 1},...,x_{r}^{\pm 1}]$. Thus Lemma \ref{l1b} applied to $R=\bold k[x_{s+1}^{\pm 1},...,x_{r}^{\pm 1}]$ implies that for some $g,f$, the module $gM[1/f]$ is finite over the algebra $\mathcal D_{s-1,q}(\bold k[x_{s}^{\pm 1},...,x_{r}^{\pm 1}][1/f])$

Since $gM$ is a $\mathcal D_{r,q}(\bold k)$-module, it follows that for all $\bold m$, the module $gM[1/f_{\bold m}]$ is finite over $\mathcal D_{s-1,q}(\bold k[x_{s}^{\pm 1},...,x_{r}^{\pm 1}][1/f_{\bold m}])$. Let $v_{\bold m,i}\in gM$, $1\le i\le \ell_{\bold m}$ be generators of $gM[1/f_{\bold m}]$. 

By Lemma \ref{l10}, there exist $\bold m_1,...,\bold m_k$, 
such that the system $f_{\bold m_j}=0$, $1\le j\le k$ has no solutions. Let $f_j:=f_{{\bold m}_j}$, $w_{ji}:=v_{\bold m_j,i}$, $\ell_j:=\ell_{\bold m_j}$.

We claim that $w_{ji}$, $1\le j\le k$, 
$1\le i\le \ell_{j}$ are generators of $gM$ over $\mathcal D_{s-1,q}(\bold k[x_{s}^{\pm 1},...,x_{r}^{\pm 1}])$. Indeed, for every $u\in M$ and each $j$ there are $N_j$ and $D_{ji}\in \mathcal D_{s-1,q}(\bold k[x_{s}^{\pm 1},...,x_{r}^{\pm 1}])$ such that 
$$
f_{j}^{N_j}u=\sum_{i=1}^{\ell_j} D_{ji}w_{ji}.
$$
Pick $h_1,...,h_k\in \bold k[x_{s}^{\pm 1},...,x_{r}^{\pm 1}]$
such that $\sum_{j=1}^k h_jf_j^{N_j}=1$ (they exist by the Nullstellensatz since the system $f_j=0$, $1\le j\le k$ has no solutions). Then 
$$
u=\sum_{j=1}^k h_jf_{j}^{N_j}u=\sum_{j=1}^k\sum_{i=1}^{\ell_j} h_jD_{ji}w_{ji},
$$
as claimed.
\end{proof} 

Now we are ready to prove Theorem \ref{t2}. The proof of part (i) is by induction in $s$, starting from the trivial case $s=0$. Suppose the result is known for $s-1$ with $0<s\le r$ and let's prove it for $s$. So assume that $M$ is a holonomic $\mathcal D_{r,q}(\bold k)$-module which is finite over $\mathcal D_{s,q}(\bold k[x_{s+1}^{\pm 1},...,x_{r}^{\pm 1}])$. Since the GK dimension of $M$ is $r$, it cannot contain a free 
submodule over $\mathcal D_{s,q}(\bold k[x_{s+1}^{\pm 1},...,x_{r}^{\pm 1}])$ (as this algebra has GK dimension $r+s>r$). Thus by Corollary \ref{c1d} there is 
$g_1\in {\rm Sp}(2s,\Bbb Z)$ for which $g_1M$ is 
finite over $\mathcal D_{s-1,q}(\bold k[x_{s}^{\pm 1},...,x_{r}^{\pm 1}])$. By the induction assumption, there is $g_2\in {\rm Sp}(2s-2,\Bbb Z)\subset {\rm Sp}(2s,\Bbb Z)$ such that $g_2g_1M$ is finite as a $\bold k[x_1^{\pm 1},...,x_r^{\pm 1}]$-module. This completes the induction step, with $g=g_2g_1$.  

Part (iii) is a special case of (i) with $s=r$. 

Finally, (ii) follows from (i) and Proposition \ref{freeness}.

\subsection{$N$-curvature}
If $q$ is a primitive $N$-th root of $1$, then there is a notion of 
$N$-{\bf curvature} of a flat $q$-connection, which is a $q$-deformation of the notion of $p$-curvature of a usual flat connection in characteristic $p$ (see e.g. \cite{EV}, Section 6). Namely, the $N$-curvature of $\nabla$ is the collection of commuting operators $C_i:=\nabla_i^N$. 
These operators are just $\bold k(x_1,...,x_r)$-linear automorphisms 
of $\bold k(x_1,...,x_r)^n$, namely, 
$$
C_i=\prod_{k=N-1}^{0} T_i^k(a_i)=T_i^{N-1}(a_i)...T_i(a_i)a_i.
$$

As for usual connections, since $[\nabla_i,C_j]=0$, we have $[C_i,C_j]=0$ and the coefficients of the characteristic polynomial $\chi_{\bold t}$ of $\sum_{j=1}^r t_jC_j$ 
are preserved by all $T_i$, hence belong to $\bold k(X_1,...,X_r)[t_1,...,t_r]$.

\subsection{Lagrangian flat $q$-connections} Let $q$ be a primitive $N$-th root of unity. Again parallel to the differential case, let us say that a flat $q$-connection $\nabla$ is {\bf separable} if the eigenvalues of its $N$-curvature operators $C_j$ lie in the separable closure $\overline{\bold k(X_1,...,X_r)}_s$. E.g., every flat connection is separable in characteristic zero, or if it has rank $n<p$ in characteristic $p$. Let $\nabla$ be separable and suppose $v$ is a common eigenvector of $C_j$, i.e. $C_jv=\Lambda_j v$, $\Lambda_j\in \overline{\bold k(X_1,...,X_r)}_s.$  Let us say that $\nabla$ is {\bf Lagrangian} if for each such $v$ we have 
$$
\frac{X_j}{\Lambda_i}\frac{\partial \Lambda_i}{\partial X_j}=\frac{X_i}{\Lambda_j}\frac{\partial \Lambda_j}{\partial X_i}.
$$ 

The motivation for this terminology is similar to the differential case. Let $\nabla$ be a regular flat $q$-connection, i.e., $a_i\in GL_n(\bold k[x_1^{\pm 1},...,x_r^{\pm 1}])$. 
Consider the $\mathcal D_{r,q}(\bold k)$-module 
$M_{\nabla}$ 
generated by $f_1,...,f_n$ with defining relations $\nabla_i \bold f=\bold f$, where $\bold f=(f_1,...,f_n)^T$. Let ${\rm supp}\nabla$ be the support of $M_{\nabla}$ as a $Z$-module. It is clear that  ${\rm supp}\nabla$ is given in the coordinates $X_i,P_i$ by the equations
$\chi_{\bold t}(\sum_{j=1}^r t_jP_j)=0$
$\forall \bold t=(t_1,...,t_r)\in \overline{\bold k}^r$, i.e., consists of $(x_1^N,...,x_r^N,\Lambda_1,...,\Lambda_r)\in \Bbb G_m^{2r}$ such that 
$\sum_{j=1}^r t_j\Lambda_j$ is an eigenvalue of $\sum_{j=1}^r t_jC_j(x_1,...,x_r)$ for all $\bold t$. Thus, we obtain 

\begin{proposition} A flat $q$-connection $\nabla$ is Lagrangian if and only if ${\rm supp}\nabla$ is a Lagrangian subvariety of ${\rm Spec}(Z)$. 
\end{proposition} 

\subsection{Liftable $q$-connections} As in the differential case, a separable flat $q$-connection, even one of rank $1$, need not be Lagrangian in general. 

\begin{example}\label{count1a} Let ${\rm ord}(q)=N$ and $\nabla$ be the flat $q$-connection of rank $1$ on $(\bold k^\times)^2$ with 
$$
\nabla_1=T_1^{-1},\ \nabla_2=T_2^{-1}x_1^N.
$$  
Then the $N$-curvature of $\nabla$ is $C_1=1,C_2=X_1^N$, so $\nabla$ is not Lagrangian. 
\end{example}

However, the Lagrangian property holds under some assumptions.   
Namely, suppose ${\rm ord}(q)=N$. 
Let us say that a flat $q$-connection $\nabla$ defined over $\bold k$ is {\bf liftable} if it is the reduction to $\bold k$ of a flat $\widetilde q$-connection defined over 
some local Artinian ring $S$ with residue field $\bold k$ in which $\widetilde q^N\ne 1$. 

\begin{proposition}\label{l1bb} Let $n$ be a positive integer which is $<p$ if $p={\rm char}(\bold k)>0$. Let $q$ be a primitive $N$-th root of unity in $\bold k$. Then a liftable $q$-connection of rank $n$ over $\bold k$ is Lagrangian. 
\end{proposition} 

\begin{proof} 
It suffices to prove the statement for $r=2$. In this case we have 
$\nabla=(\nabla_x,\nabla_y)$ with 
$\nabla_x=T_x^{-1}a(x,y)$ and $\nabla_y=T_y^{-1}b(x,y)$
and lifting $\widetilde\nabla=(\widetilde\nabla_x,\widetilde\nabla_y)$ with 
$\widetilde\nabla_x=T_x^{-1}\widetilde a(x,y)$ and $\widetilde\nabla_y=T_y^{-1}\widetilde b(x,y)$, where $T_x,T_y$ shift $x,y$ by $\widetilde q$. Thus 
$$
\widetilde\nabla_x^N=T_x^{-N}\widetilde A(x,y),\ \widetilde\nabla_y^N=T_y^{-N}\widetilde B(x,y),
$$
where 
$$
\widetilde A(x,y)=\widetilde a(\widetilde q^{N-1}x,y)...\widetilde a(x,y),\ \widetilde B(x,y)=\widetilde b(x,\widetilde q^{N-1}y)...\widetilde b(x,y).
$$
Note that the reductions of $\widetilde A,\widetilde B$ to $\bold k$ 
are the components $A(x,y),B(x,y)$ of the $N$-curvature $C(\nabla)$. The zero curvature equation 
$$
\widetilde A(x,\widetilde q^Ny)\widetilde B(x,y)=\widetilde B(\widetilde q^Nx,y)\widetilde A(x,y)
$$
implies that 
$$
(\widetilde A(x,\widetilde q^Ny)-\widetilde A(x,y))\widetilde B(x,y)+[\widetilde A(x,y),\widetilde B(x,y)]=(\widetilde B(\widetilde q^Nx,y)-\widetilde B(x,y))\widetilde A(x,y).
$$
Note that for any $n$-by-$n$ matrices $P,Q$ we have 
$$
\sum_{j=0}^{m-1}{\rm Tr}(Q^j[P,Q]Q^{m-1-j}P^{\ell-1})=\sum_{j=1}^{m-1}{\rm Tr}([P,Q^m]P^{\ell-1})=0.
$$
Thus for $\ell,m\ge 1$ we have 
\scriptsize
$$
\sum_{j=0}^{m-1}{\rm Tr} \widetilde B(x,y)^j\left((\widetilde A(x,\widetilde q^Ny)-\widetilde A(x,y))\widetilde B(x,y)-(\widetilde B(\widetilde q^Nx,y)-\widetilde B(x,y))\widetilde A(x,y)\right)\widetilde B(x,y)^{m-1-j}\widetilde A(x,y)^{\ell-1}=0.
$$
\normalsize The left hand side of this equation belongs to $(\widetilde q^N-1)S$. Since $\widetilde q^N-1\ne 0$ but belongs to the maximal ideal of $S$,
this yields 
$$
\sum_{j=0}^{m-1}{\rm Tr} B(x,y)^j\left(y\tfrac{\partial A(x,y)}{\partial y}B(x,y)-x\tfrac{\partial B(x,y)}{\partial x} A(x,y)\right)B(x,y)^{m-1-j}A(x,y)^{\ell-1}=0.
$$
Since $[A,B]=0$, for $m<p$ this implies 
\begin{equation}\label{eq2}
{\rm Tr} \left(y\tfrac{\partial A(x,y)}{\partial y}B(x,y)-x\tfrac{\partial B(x,y)}{\partial x} A(x,y)\right)A(x,y)^{\ell-1}B(x,y)^{m-1}=0.
\end{equation}
Near a generic point $(x,y)$, the matrices $A,B$ have constant multiplicities $m_k$ of joint eigenvalues $(\alpha_k,\beta_k)$. Let $\bold e_k={\bold e}_k(x,y)$ be the projectors onto the corresponding generalized joint eigenspaces which commute with $A,B$. 
These projectors are polynomials of $A,B$ which can be taken to be of degree $<n$ in each variable, so \eqref{eq2} implies 
\begin{equation}\label{eq1}
{\rm Tr} \left(\alpha_k^{-1}y\tfrac{\partial A}{\partial y}-\beta_k^{-1}x\tfrac{\partial B}{\partial x} \right){\bold e}_k=0.
\end{equation}
Recall that if $e=e(t),M=M(t)$ are 1-parameter families in ${\rm Mat}_n$ 
such that $e^2=e$ and $[M,e]=0$ then ${\rm Tr}(Me')=0$: indeed, we have 
$ee'=e'(1-e)$, hence $ee'e=(1-e)e'(1-e)=0$, 
and $M=eMe+(1-e)M(1-e)$, so ${\rm Tr}(Me')=
{\rm Tr}(eMee')+{\rm Tr}((1-e)M(1-e)e')={\rm Tr}(Mee'e)+{\rm Tr}(M(1-e)e'(1-e))=0$. 
Thus ${\rm Tr}(A\partial_y \bold e_k)={\rm Tr}(B\partial_x\bold e_k)=0$. Hence equation \eqref{eq1} can be written as 
$$
\alpha_k^{-1}y\frac{\partial {\rm Tr}A{\bold e}_k}{\partial y}-\beta_k^{-1}x\frac{\partial{\rm Tr} B{\bold e}_k}{\partial x}=0,
$$
i.e., 
$$
m_k (\alpha_k^{-1}y\tfrac{\partial \alpha_k}{\partial y}-\beta_k^{-1}x\tfrac{\partial \beta_k}{\partial x})=0.
$$
Since $m_k\le n<p$ and $(N,p)=1$ if $p>0$, this is equivalent to 
$$
\frac{Y}{\alpha_k}\frac{\partial \alpha_k}{\partial Y}=\frac{X}{\beta_k}\frac{\partial \beta_k}{\partial X},
$$
where $X=x^N$, $Y=y^N$, 
as claimed. 
\end{proof} 

\begin{example} However, if ${\rm char}(\bold k)=p$ then a liftable separable 
$q$-connection of rank $n\ge p$ need not be Lagrangian.  Namely, let $\bold k$ be a finite field of characteristic $p$ and $K$ be a local field (necessarily of characteristic zero) with residue field $\bold k$ such that its ring of integers $\mathcal O_K$ contains a primitive root of unity $q$ of order $pN$. Let $L,\sigma$ be as in Example \ref{LMex}, and $\nabla$ be the $q$-connection of rank $p$ on the 2-torus $\Bbb G_m^2$ defined over $S:=\mathcal O_K/(q^N-1)^2$ by 
$$
\nabla_x=T_x^{-1}q^{NL},\ \nabla_y=T_y^{-1}x^{N}\sigma.
$$  
It is easy to check that this $q$-connection is flat. However, reduction of $\nabla$
to $\bold k$ yields the $q$-connection 
$$
\nabla_x=T_x^{-1},\ \nabla_y=T_y^{-1}x^{N}\sigma
$$  
with $N$-curvature $C_x=1,C_y=X^N\sigma^N$, which is not Lagrangian. 
\end{example}

\subsection{The $q$-analog of Bitoun's theorem} 

Let $\bold k$ be an algebraically closed field of characteristic zero. 
Suppose that $q\in \bold k$ is not a root of unity. Let $M$ be a finitely generated $\mathcal D_{r,q}(\bold k)$-module. Analogously to $D_r(\bold k)$-modules, $M$ is defined over a finitely generated subring $R\subset \bold k$ containing $q^{\pm 1}$, and we can pick a finitely generated $\mathcal D_{r,q}(R)$-module $M_R$ with an isomorphism $\bold k\otimes_R M_R\cong M$. 
Alternatively, one may take $R$ to be a finitely generated subalgebra containing $q^{\pm 1}$ over a subfield $L\subset\bold k$, such that $q$ is transcendental over $L$. For every maximal ideal $\mathfrak m$ of $R$ with residue field $k_{\mathfrak m}$ let $M_{\mathfrak m}$ be the finitely generated $\mathcal{D}_{r,q}(k_{\mathfrak m})$-module $k_{\mathfrak m}\otimes_R M_R$. 

Let us say that a maximal ideal $\mathfrak m\subset R$ is a $q$-\textbf{torsion point of
order $N$} if the image of $q$ in $R/\mathfrak m$ has multiplicative
order $N$. 

If $\mathfrak{m}$ is a $q$-torsion point then we may define the support ${\rm supp}(M_{\mathfrak{m}})$ of $M_{\mathfrak{m}}$ as a $Z$-module, which by analogy with the differential case we will call the $p$-{\bf support} of $M$ at $\mathfrak{m}$ (even though the field $k_{\mathfrak{m}}$ may have the same characteristic as $\bold k$, i.e. characteristic zero). 

\begin{theorem}\label{t6a} If $M\ne 0$ is a holonomic $\mathcal D_{r,q}(\bold k)$-module then ${\rm supp}(M_{\mathfrak{m}})$ is Lagrangian for a generic $q$-torsion point $\mathfrak m\in {\rm Spec} R$.
\end{theorem} 

\begin{example} If $\bold k=\overline{\Bbb Q}$ then for each maximal ideal $\mathfrak m\in {\rm Spec} R$, $R/\mathfrak m$ is a finite field, so $\mathfrak{m}$ is a torsion point. Thus Theorem \ref{t6a} holds for all $\mathfrak{m}$ lying over $p$ for almost all primes $p$.  
\end{example} 

\begin{proof} By Theorem \ref{t2}, it suffices to consider the case when $M$ is a finite $\bold k[x_1^{\pm 1},...,x_r^{\pm 1}]$-module. But then the result follows from Proposition \ref{l1bb}. 
\end{proof} 

\begin{remark} \label{rankbound}
Suppose that $\operatorname{char}(\bold k)=p>0$, $q$ has infinite
order, and $M\ne0$ is holonomic. Put
$\mathcal O=\bold k[x_1^{\pm1},\ldots,x_r^{\pm1}]$.
If $gM$ is finite of rank $n<p$ over $\mathcal O$ for some
$g\in\operatorname{Sp}(2r,\Bbb Z)$, then the conclusion of
Theorem~\ref{t6a} holds. Indeed, Proposition~\ref{freeness} makes
$gM$ free, and the same proof applies using Proposition~\ref{l1bb}.

More explicitly, suppose that, after a symplectic change of variables,
$M$ has $\mathcal D_{r,q}(\bold k)$-generators $v_1,\ldots,v_s$
satisfying, for $1\le a\le s$ and $1\le i\le r$,
\[
 \left(\sum_{j=0}^{d_{ai}}c_{aij}(x)T_i^j\right)v_a=0,
 \qquad d_{ai}\ge1,\quad c_{aij}\in\mathcal O,\quad
 c_{ai0},c_{ai,d_{ai}}\in\mathcal O^\times.
\]
Then it suffices that
\[
 p>\sum_{a=1}^s\prod_{i=1}^r d_{ai};
 \qquad\text{in particular, }p>s d^r\text{ suffices if all }d_{ai}\le d.
\]
Indeed, translating these recurrences and using their invertible
endpoint coefficients shows that the vectors
$T_1^{e_1}\cdots T_r^{e_r}v_a$, $0\le e_i<d_{ai}$, generate $M$
over $\mathcal O$, giving the stated rank bound.
\end{remark}

It is useful to make a slight generalization of Theorem \ref{t6a}.
Let $\Lambda=(\lambda_{ij})$ be a skew-symmetric integer $s$-by-$s$ matrix. Denote by $\mathcal D_{\Lambda,q}=\mathcal D_{\Lambda,q}(\bold k)$ the $\bold k$-algebra with invertible generators $X_1,...,X_s$ and defining relations 
$$
X_iX_j=q^{\Lambda_{ij}}X_jX_i.
$$
Let $K:={\rm Ker}\Lambda\subset \Bbb Z^s$ be the subgroup of vectors $\ell$ such that $\Lambda\ell=0$. 
This is a saturated subgroup, so it admits a complement $L$ such that $\Bbb Z^s=K\oplus L$. Picking a basis of $L$, we see that $\mathcal D_{\Lambda,q}$ is isomorphic 
to $\mathcal D_{\Lambda',q}$ where $\Lambda'=\Lambda_*\oplus 0$ and $\Lambda_*$ is nondegenerate. 
Thus we may assume that $\Lambda$ is in this form to start with, i.e., $\mathcal D_{\Lambda,q}=\mathcal D_{\Lambda_*,q}\otimes \bold k[\Bbb Z^{s-2r}]$, where $2r$ is the size of $\Lambda_*$ (i.e., the rank of $\Lambda$). 
Moreover, pick a symplectic $\Bbb Q$-basis $v_i$ for $\Lambda_*$ and set $Y_j:=\prod_{i=1}^{2r} X_i^{dv_{ij}}$ where $d$ is the common denominator of $v_{ij}$. Then $Y_i$ generate 
a subalgebra $\mathcal D_{r,q^{d^2}}\subset \mathcal D_{\Lambda_*,q}$.
Thus $\mathcal D_{\Lambda,q}$ is a finite extension of $\mathcal D_{r,q^{d^2}}\otimes \bold k\Bbb Z^{s-2r}$ (isogeny of quantum tori). 

At a $q$-torsion point of order $N$, the full center $Z$ of
$\mathcal D_{\Lambda,q}$ has spectrum isomorphic to a torus $\Bbb G_m^s$.
It carries the Poisson bracket induced by varying $q$, normalized by
\[
 \{X_i^N,X_j^N\}=\lambda_{ij}X_i^N X_j^N.
\]
(However, the central elements $X_i^N$ need not generate $Z$).

If $\Lambda$ is degenerate, this bracket is also, but one can still talk about Lagrangian subvarieties of $\Bbb G_m^s$ - coisotropic subvarieties of dimension $r$. Also by analogy with the case of $\mathcal D_{r,q}$ one can define the notion of a holonomic module - a $\mathcal D_{\Lambda,q}$-module of GK dimension $\le r$. 
The fact that $\mathcal D_{\Lambda,q}$ is a finite extension of $\mathcal D_{r,q^{d^2}}\otimes \bold k\Bbb Z^{s-2r}$ implies that such modules have finite length, and Theorem \ref{t6a} implies 

\begin{theorem}\label{t6b} If $M\ne 0$ is a holonomic $\mathcal D_{\Lambda,q}(\bold k)$-module then ${\rm supp}(M_{\mathfrak{m}})$ is Lagrangian for a generic $q$-torsion point $\mathfrak m\in {\rm Spec} R$.
\end{theorem}   

\subsection{$p$-supports of holonomic modules over quantum cluster algebras} \label{clusalg}
Now we would like to apply Theorem \ref{t6b} to quantum cluster algebras as defined by Berenstein and Zelevinsky (\cite{BZ}). We first recall the basic definitions 
of the theory of cluster algebras and Poisson cluster algebras (\cite{FZ,GSV}).
We will first need a number-theoretic lemma. 

\begin{lemma}\label{ordroo} 
For every positive integer $d$, the $q$-torsion points of
order coprime to $d$ are Zariski dense in $\operatorname{Spec}R$.
\end{lemma}

\begin{proof}  If $R$ is a finitely generated 
subalgebra over a subfield $L\subset \bold k$ containing $q^{\pm 1}$ with $q$ transcendental over $L$, then density easily follows from the
infinitude of roots of unity of order coprime to $d$ in $\overline L$. 

On the other hand, if $R$ is finitely generated over $\Bbb Z$ (the case needed to handle algebraic $q$), the lemma is less trivial, but follows e.g. from \cite{P}, Theorem 7. Let us give a direct proof for reader's convenience. 

We may assume that $d>1$. Fix $0\ne f\in R$. Our job is to show that 
$D(f)\subset {\rm Spec} R$ contains a $q$-torsion point of order coprime to $d$. 

Choose a specialization
$\varphi:R[f^{-1}]\to K$, where $K$ is a number field, with
$a:=\varphi(q)$ non-torsion. Namely, if $q$ is algebraic, any closed point $\varphi$ of $\operatorname{Spec}(R[f^{-1}]\otimes\Bbb Q)$ works. Otherwise,
first choose $a\in \Bbb Q\setminus \lbrace 0,\pm 1\rbrace$ 
in a nonempty open subset of the image of $q: \operatorname{Spec}(R[f^{-1}]\otimes\Bbb Q)\to \Bbb A^1_{\Bbb Q}$, and then pick a closed point $\varphi\in q^{-1}(a)$ 
(i.e., such that $\varphi(q)=a$).

Let $\Sigma$ be the set of primes $\ell$ dividing $d$. Kummer theory and \cite[Lemma~5]{Iw} give an open subgroup
\[
 G:=\operatorname{Gal}\bigl(
 K(\mu_{\ell^\infty},a^{1/\ell^\infty}:\ell\in\Sigma)/K\bigr)
 \ \subset\ \prod_{\ell\in\Sigma}
 (\Bbb Z_\ell\rtimes\Bbb Z_\ell^\times).
\]
Indeed, bounded divisibility of $a$ in $K(\mu_\infty)$ makes the
Kummer kernel open, and the cyclotomic image is open as well.
For $h\gg0$, $G$ therefore contains an element with components
$(0,1+\ell^h)$. Its first coordinate records the action on compatible
radicals, and its second the action on roots of unity.
Chebotarev's theorem (\cite[Theorem~8.31]{Milne}), applied to its conjugacy
class in
${\rm Gal}(K(\mu_{\ell^{h+1}},a^{1/\ell^h}:\ell\in\Sigma)/K)$
gives a positive-density set of primes $\mathfrak p$ of $K$ with
residue cardinality $Q$ such that
\[
 v_\ell(Q-1)=h,\qquad
 \overline a\in(\Bbb F_Q^\times)^{\ell^h}
 \quad(\ell\in\Sigma),
\]
where $\overline a\in \Bbb F_Q^\times$ is the reduction of $a$. 
Thus $\operatorname{ord}(\overline a)$ is coprime to $d$.
Outside finitely many primes, $\varphi$ can be reduced modulo
$\mathfrak p$; its kernel gives the required $q$-torsion point in
$D(f)$.
\end{proof} 

We will say that a statement holds for {\bf generic $q$-torsion points of order
coprime to $d$} if it holds at every such point in some nonempty
Zariski open subset of $\operatorname{Spec}R$.

Let $m\ge n$, and $B\in {\rm Mat}_n(\Bbb Z)$ be skew-symmetrizable, i.e., there exist positive integers $d_i$, $1\le i\le n$, such that for $D={\rm diag}(d_1,...,d_n)$ the matrix $DB$ is skew-symmetric. Let $M$ be an integer $(m-n)\times n$ matrix, and let  
$\widetilde B=(b_{ij})=\binom{B}{M}$ be the corresponding $m\times n$ matrix. 
The matrix $\widetilde B$ is then called a {\bf seed} (or {\bf exchange}) matrix.

Given a seed matrix $\widetilde B$ and $1\le k\le n$, we can define the {\bf cluster mutation} 
of $\widetilde B$ at $k$ by 
$$
\mu_k(\widetilde B)=\widetilde B'=(b_{ij}'),\ b_{ij}':=\begin{cases} -b_{ij} \text{ if }i=k\text{ or }j=k\\
b_{ij}+\frac{|b_{ik}|b_{kj}+b_{ik}|b_{kj}|}{2}\text{ else }\end{cases}
$$
This is accompanied by the change of generators $X_1,...,X_m$ in the field $\mathcal F=\Bbb Q(X_1,...,X_m)$ (the initial {\bf cluster} in $\mathcal F$ corresponding to the seed $\widetilde B$)
which changes only the generator $X_k$ to 
\begin{equation}\label{clutch}
X_k':=\frac{\prod_{i: b_{ik}>0}X_i^{b_{ik}}+\prod_{i: b_{ik}<0}X_i^{-b_{ik}}}{X_k}
\end{equation} 
(i.e., $X_j'=X_j$ if $j\ne k$). One says that $X_1',...,X_m'$ is the cluster in $\mathcal F$ 
corresponding to the seed $\widetilde B'$, and $X_i'$ are called the {\bf cluster variables} of this cluster. This way one can create more and more clusters 
and seeds, by successively performing mutations at various $k$ (noting however that $\mu_k^2={\rm id}$). Note that the last $m-n$ variables are not changed under mutation, 
which is why they are called {\bf frozen variables}. 

Now one defines the {\bf upper cluster algebra} $\bold U(\widetilde B)$ to be the 
intersection of all the subalgebras $\bold U_C$ generated inside $\mathcal F$ by the cluster variables and their inverses for all clusters $C$ that can be obtained by successive mutations from the initial cluster $C_0=(X_1,...,X_m)$. 

It is nontrivial even that $\bold U(\widetilde B)$ contains anything beyond constants. But it is in fact quite large due to the {\bf Laurent phenomenon} \cite{FZ} which is what really makes the cluster algebras tick. Namely, $\bold U(\widetilde B)$ contains the {\bf cluster algebra} $\bold A(\widetilde B)$ generated by the cluster variables of all clusters (but without inverses) and also inverses of the frozen variables. 

Now suppose $\Lambda\in {\rm Mat}_m(\Bbb Z)$ is a skew-symmetric matrix. 
We say that $\Lambda$ is compatible to an integer $m\times n$-matrix $\widetilde B$, or that 
$(\Lambda,\widetilde B)$ is a {\bf compatible pair}, if 
$$
\sum_{k=1}^m b_{ki}\lambda_{kj}=\delta_{ij}d_j, 
$$
where $d_i$ for $1\le i\le n$ are positive integers, i.e, $\widetilde B^T\Lambda=(D|0)$ where $D={\rm diag}(d_1,...,d_n)$. In this case, $\widetilde B$ has full rank $n$, and 
$DB=\widetilde B^T \Lambda \widetilde B$ is skew-symmetric, so $\widetilde B$ is a seed matrix.  

Now let $(\Lambda,\widetilde B)$ be a compatible pair and endow $\mathcal F$ 
with a Poisson bracket 
$$
\lbrace X_i,X_j\rbrace=\lambda_{ij}X_iX_j,
$$ 
i.e. the log-canonical bracket in the initial cluster $C_0$ attached to the matrix $\Lambda$. 
Then we have 
$$
\lbrace{X_i',X_j'\rbrace}=\lambda_{ij}'X_i'X_j'
$$ 
for an appropriate 
matrix $\Lambda'=(\lambda_{ij}')$ compatible to $\widetilde B'=\mu_k(\widetilde B)$, and the pair $(\Lambda',\widetilde B')$ is called the mutation at $k$ of $(\Lambda,\widetilde B)$ and denoted $\mu_k(\Lambda,\widetilde B)$. Namely, 
$$
\Lambda'=E_k^T\Lambda E_k
$$
where $(E_k)_{ij}=\delta_{ij}$ unless $j=k$, $(E_k)_{kk}=-1$, and $(E_k)_{ik}=\max(0,b_{ik})$, $i\ne k$.
This implies that the algebra $\bold U(\widetilde B)$ is naturally a Poisson algebra. This algebra is called a {\bf cluster Poisson algebra} and denoted
$\bold U(\Lambda,\widetilde B)$. 

A great feature of log-canonical Poisson brackets is that they are easy to quantize: 
the quantization is given by the quantum torus algebra $\mathcal D_{\Lambda,q}$. 
Remarkably, it turns out (\cite{BZ}) that this process is compatible with cluster mutations (due to the {\bf quantum Laurent phenomenon}) which gives rise to the {\bf upper quantum cluster algebra} $\bold U_q(\Lambda,\widetilde B)$, the intersection of the algebras $\mathcal D_{\Lambda_C,q}$ for various clusters $C$ inside their common noncommutative fraction field $\mathcal F_q$.

In many interesting cases, $\bold U_q(\Lambda,\widetilde B)$ is known to be a flat deformation of $\bold U(\Lambda,\widetilde B)$, so we will henceforth assume this. Indeed, this is so, for example, when $\bold A(\widetilde B)=\bold U(\widetilde B)$, which holds for all
locally acyclic cluster algebras, in particular acyclic ones
(\cite[Theorem~2]{MuAU}). Namely, let $\mathcal U_v$ be the upper quantum cluster algebra over
$\bold k[v^{\pm1}]$, where $v$ is an indeterminate and the commutation parameter $q$ is $v^2$. By \cite[Corollary~3.6]{GLS},
if $\bold A(\widetilde B)=\bold U(\widetilde B)$ then $\mathcal U_v/(v-1)\mathcal U_v\cong\bold U(\widetilde B).$
Moreover, $\mathcal U_v$ is torsion-free over the principal ideal
domain $\bold k[v^{\pm1}]$, hence flat. 
 
If the algebra $\bold U(\widetilde B)$ for a given seed matrix $\widetilde B$ is finitely generated and $\widetilde B$ has full rank, then ${\rm Spec}\bold U(\widetilde B)$ is an irreducible normal affine variety $Y(\widetilde B)$ equipped with a smooth open subset $Y^\circ(\widetilde B)\subset Y(\widetilde B)$ with complement of codimension $\ge 2$ covered by toric charts $Y_C(\widetilde B)$ labeled by various clusters $C$, so that the clutching maps between charts corresponding to clusters connected by a single mutation are given by \eqref{clutch} (\cite{MNTY}, Proposition 4.6). Note that in this case $Y^\circ(\widetilde B)$ is covered by a finite subset of such charts, since it is quasi-compact. 

If in addition $Y(\widetilde B)=Y(\Lambda,\widetilde B)$ is equipped 
with a cluster Poisson bracket defined by a skew-symmetric matrix $\Lambda$ of rank $2r$ then the toric charts carry log-canonical Poisson structures of rank $2r$, so the Poisson structure of $Y(\Lambda,\widetilde B)$
has constant rank $2r$ on $Y^\circ(\Lambda,\widetilde B)$ (outside of this set the rank can drop). More precisely, by Theorem A of \cite{MNTY}, the Poisson bracket has maximal rank precisely on the smooth locus of $Y(\Lambda,\widetilde B)$, which is the orbit of a single symplectic leaf under the torus $(\Bbb C^\times)^{\dim {\rm Ker}\widetilde B^T}$ which acts naturally on $Y(\Lambda,\widetilde B)$ by Poisson automorphisms. 
Then the algebra $\bold U_q(\Lambda,\widetilde B)$ is a quantization of the Poisson variety $Y(\Lambda,\widetilde B)$. 

In this case, let ${\mathcal C}_q(\Lambda,\widetilde B)$ be the category of $\bold U_q(\Lambda,\widetilde B)$-modules, and ${\mathcal C}_q^0(\Lambda,\widetilde B)$ be its Serre subcategory of modules $M$ for which $M_C:=\mathcal D_{\Lambda_C,q}\otimes_{\bold U_q(\Lambda,\widetilde B)}M$ vanishes for each cluster $C$. Let $\overline {\mathcal C}_q(\Lambda,\widetilde B):=
{\mathcal C}_q(\Lambda,\widetilde B)/\mathcal C_q^0(\Lambda,\widetilde B)$ be the Serre quotient. 
For each $M\in \overline {\mathcal C}_q(\Lambda,\widetilde B)$, the localization $M_C$ is well defined. 

 Let us say that $M\in \overline {\mathcal C}_q(\Lambda,\widetilde B)$ is 
holonomic if $M_C$ is holonomic for all clusters $C$. If $q$ is specialized to a $q$-torsion point $\mathfrak m\subset R$, for every $C$ consider the $p$-support ${\rm supp}(M_{C,\mathfrak m})\subset Y_C^{(1)}(\widetilde B)$, where $Y_C^{(1)}(\widetilde B)={\rm Spec}(Z_C)$, 
the spectrum of the center $Z_C$ of $\mathcal D_{\Lambda_C,q}$. It is easy to see that these supports glue together into a closed subvariety ${\rm supp}(M_{\mathfrak m})$ of $Y^{\circ (1)}(\widetilde B)$, the cluster Poisson variety glued from the toric charts $Y_C^{(1)}(\widetilde B)$. We call this subvariety the $p$-support of $M$. 

The charts $\operatorname{Spec}Z_C$ are glued by the Poisson maps
induced by the specialized quantum mutations. For adjacent quantum
tori $A,A'$ with exchanged variables $x,x'$ and $N=\operatorname{ord}(q)$,
the identity
\[
 A[(x^{\prime N})^{-1}]=A'[(x^N)^{-1}]
\]
identifies their localized centers and modules. The notation
$Y^{\circ(1)}(\widetilde B)$ denotes this gluing.

Theorem \ref{t6b} immediately implies the following theorem. 

\begin{theorem} \label{t6c} 
If $M\ne0$ is a holonomic object of
$\overline{\mathcal C}_q(\Lambda,\widetilde B)$, then
$\operatorname{supp}(M_{\mathfrak m})$ is Lagrangian for generic
$q$-torsion points $\mathfrak m\in\operatorname{Spec}R$ whose order is
coprime to each of $d_1,\ldots,d_n$.
\end{theorem}

\begin{proof}
After localizing $R$ if needed, choose a finite classical cluster-chart cover
over $R$. If $N$ is coprime to every $d_i$, the quantum-binomial
identity identifies the central $N$-th powers, up to constant
rescaling over a residue-field extension, with the classical cluster
coordinates (cf.\ \cite[Proposition~4.4]{NTY}). The resulting maps from
the full central tori to the classical tori glue, and the inverse
image of each classical chart is its full-center chart. Thus the
same finite collection covers every allowed specialization.
Apply Theorem~\ref{t6b} to these finitely many chart modules and
localize $R$ once more so that their specialized supports
are Lagrangian. Since the supports agree on overlaps, their union
is Lagrangian.
\end{proof}

\begin{remark} 1. Lemma \ref{ordroo} guarantees that such 
torsion points exist, and moreover they are dense in ${\rm Spec} R$.

2. Theorem \ref{t6c} straightforwardly generalizes to the setting when some of the frozen variables are not inverted.

3. It would be interesting to extend Theorem~\ref{t6c} to suitably defined holonomic $\bold U_q$-modules whose localization
at every cluster is zero. The methods of this paper are too crude to do this in general, but one can possibly treat the case of multiplicative Higgs branches (quasi-Hamiltonian reductions from $q$-differential operators on a vector space, such as multiplicative quiver varieties)
similarly to how we proposed to treat ordinary Higgs branches in Subsection \ref{symsing}.  
\end{remark}


\begin{thebibliography}{999999}

\bibitem[AGM]{AGM} A. Aizenbud, D. Gourevitch, A. Minchenko,
Holonomicity of relative characters and applications
to multiplicity bounds for spherical pairs, arXiv:1501.01479, Sel. Math. New Ser., Volume 22, pages 2325--2345, (2016), DOI 10.1007/s00029-016-0276-4.

\bibitem[Be]{Be} A. Beauville, 
Symplectic singularities, Invent. Math. 139 (2000), 541--549.

\bibitem[BFG]{BFG} R. Bezrukavnikov, M. Finkelberg, V. Ginzburg, Cherednik algebras and Hilbert schemes in characteristic p (with an appendix by Pavel Etingof) Representation Theory, 10 (2006), 254--298.

\bibitem[BZ]{BZ} A. Berenstein and A. Zelevinsky, Quantum cluster algebras, Adv. Math. 195 (2005), 405-455.

\bibitem[BK1]{BK} R. Bezrukavnikov, D. Kaledin, Fedosov quantization in the algebraic context. Moscow Math. J. 4 (2004), 559-592, arXiv:math/0309290.

\bibitem[BCHM]{BCHM} Birkar, C., Cascini, P., Hacon, C., McKernan, J.: Existence of minimal
models for varieties of log general type, J. Amer. Math. Soc. 23 (2010), 405--468.

\bibitem[B]{B} T. Bitoun, On the $p$-supports of a holonomic D-module, Inventiones Mathematicae, Volume 215, Issue 3, 2019, p. 779--818, arXiv:1012.4081. 

\bibitem[vdB]{vdB}  M. van den Bergh, On involutivity of $p$-support, International Mathematics Research Notices, Volume 2015, Issue 15, 2015, Pages 6295--6304, \url{https://doi.org/10.1093/imrn/rnu116}, arXiv:1309.6677.

\bibitem[EV]{EV} P. Etingof and A. Varchenko, p-curvature of periodic pencils of flat connections, arXiv:2401.05652. 

\bibitem[FZ1]{FZ} S. Fomin and A. Zelevinsky, Cluster algebras I: Foundations, J. Amer. Math. Soc. 15 (2002), 497-529.

\bibitem[FZ2]{FZ2} S. Fomin and A. Zelevinsky, The Laurent phenomenon, Adv. in Appl. Math. 28 (2002), 119-144. 

\bibitem[GLS]{GLS} C. Gei{\ss}, B. Leclerc, and J. Schr\"oer,
Quantum cluster algebras and their specializations,
J. Algebra 558 (2020), 411--422,
\url{https://doi.org/10.1016/j.jalgebra.2019.04.033},
arXiv:1807.09826. 


\bibitem[GSV]{GSV} M. Gekhtman, M. Shapiro, and A. Vainshtein, Cluster algebras and Poisson geometry, Mathematical Surveys and Monographs, 
Volume 167, AMS, 2010. 

\bibitem[Iw]{Iw} K. Iwasawa,
A note on Kummer extensions,
J. Math. Soc. Japan 5 (1953), nos.~3--4, 253--262;
pp.~258--259,
\url{https://doi.org/10.2969/jmsj/00530253}.

\bibitem[Ka]{K1} D. Kaledin, Symplectic singularities from the Poisson point of view. J. Reine Angew. Math. 600, 135-156 (2006), arXiv:math/0310186.

\bibitem[K]{K} M. Kontsevich, Holonomic $D$-modules in positive characteristic, Japan. J. Math. 4, 1--25 (2009)
DOI: 10.1007/s11537-009-0852-x, arXiv:1010.2908.

\bibitem[La]{La} T.Y.Lam, Serre's problem on projective modules, Springer monographs in
mathematics, 2006.

\bibitem[L]{L} G. Lonergan, Steenrod Operators, the Coulomb Branch and the Frobenius Twist, I, 
arXiv:1712.03711, Compositio Mathematica 157 (11), 2494-2552, 2021. 

\bibitem[Lo1]{Lo1} I. Losev, Deformations of symplectic singularities and orbit method for semisimple Lie algebras, Selecta Math., Volume 28, (2022), arXiv:1605.00592.

\bibitem[Lo2]{Lo2} I. Losev, Bernstein inequality and holonomic modules, with an appendix by P. Etingof and I. Losev, Advances in Mathematics, v. 308, 2017, p. 941--963, arXiv:1501.01260.

\bibitem[Lo3]{Lo3} I. Losev, Derived equivalences for symplectic reflection algebras, Int. Math. Res. Not. IMRN 2021, no. 1, 444--474., arXiv:1704.05144.

\bibitem[Milne]{Milne} J. S. Milne, Algebraic Number Theory,
version 3.08 (2020),
\url{https://www.jmilne.org/math/CourseNotes/ANT.pdf}.

\bibitem[Mu]{MuAU} G. Muller,
$\mathcal A=\mathcal U$ for locally acyclic cluster algebras,
SIGMA 10 (2014), 094, 8 pp.,
\url{https://doi.org/10.3842/SIGMA.2014.094},
arXiv:1308.1141. 

\bibitem[MNTY]{MNTY} G. Muller, B. Nguyen, K. Trampel, M. Yakimov, Poisson geometry and Azumaya loci of cluster algebras, Advances in Mathematics
Volume 453, 2024, 45pp., arXiv:2209.11622.

\bibitem[N1]{N1} Y. Namikawa,  A finiteness theorem on symplectic singularities
Compositio Mathematica, 
2016, Vol. 152 (6), 
pp. 1225-1236,  arXiv:1411.5585.

\bibitem[N2]{N2} Y. Namikawa, Poisson deformations of affine symplectic varieties, 
Duke Math. J. 156(1): 51-85, arXiv:math/0609741.

\bibitem[N3]{N3} Y. Namikawa, On deformations of $\Bbb Q$-factorial symplectic varieties,
J. Reine Angew. Math. 599 (2006) 97--110.

\bibitem[N4]{N4} Y. Namikawa,  Poisson deformations and birational geometry,J. Math. Sci. Univ. Tokyo Vol. 22 (2015), No. 1, Page 339–359.

\bibitem[NTY]{NTY} B. Nguyen, K. Trampel, and M. Yakimov,
Root of unity quantum cluster algebras and discriminants, Journal of the London Mathematical Society
Volume 111, Issue 1, https://doi.org/10.1112/jlms.70060,
arXiv:2012.02314v4.

\bibitem[P]{P} A. Perucca, Prescribing valuations of the order of a point in the reductions of abelian varieties and tori, J. Number Theory 129 (2009), no. 2, 469-476.

\bibitem[S]{S} C. Sabbah, Systemes holonomes d'\'equations aux $q$-diff\'erences, Proceedings of the International Conference on D-Modules and Microlocal Geometry at the University of Lisbon (Portugal), October 29--November 2, 1990, De Gruyter, 1993, p.125-147, \url{https://perso.pages.math.cnrs.fr/users/claude.sabbah/articles/eqdfbis.pdf}


\begingroup\color{aiSuggestedPurple}
\par\endgroup
\end{thebibliography}
\end{document}